\documentclass[11pt,reqno]{amsart}
\usepackage{enumerate, latexsym, amsmath, amsfonts, amssymb, amsthm, color}
\def\pmod #1{\ ({\rm{mod}}\ #1)}
\def\Z{\Bbb Z}

\def\Q{\Bbb Q}

\def\R{\Bbb R}

\def\l{\left}
\def\r{\right}
\def\bg{\bigg}
\def\({\bg(}
\def\){\bg)}
\def\t{\text}
\def\f{\frac}

\def\sign{{\rm sign}}

\def\ls{\leqslant}

\def\eq{\equiv}

\def\da{\delta}

\def\adj{{\rm adj}}

\def\rank{{\rm rank}}

\def\u{{\mathbf u}}
\def\v{{\mathbf v}}

\theoremstyle{plain}
\newtheorem{theorem}{Theorem}

\newtheorem{lemma}{Lemma}

\newtheorem{conjecture}{Conjecture}
\theoremstyle{definition}

\theoremstyle{remark}
\newtheorem{remark}{Remark}

\makeatletter
\@namedef{subjclassname@2020}{%
  \textup{2020} Mathematics Subject Classification}
\makeatother

\allowdisplaybreaks

 \vspace{4mm}

\begin{document}
\hbox{Preprint}
\medskip

\title
[On determinants involving $(\frac{j^2-k^2}p)$ and $(\f{jk}p)$]
{On determinants involving $(\frac{j^2-k^2}p)$ and $(\f{jk}p)$}

\author
[Deyi Chen and Zhi-Wei Sun] {Deyi Chen and Zhi-Wei Sun}

\address {(Deyi Chen) 18 Zheda Road, Xihu District, Hangzhou 310013, Zhejing Province, People's Republic of China}
\email{deyi@zju.edu.cn}

\address{(Zhi-Wei Sun, corresponding author) School of Mathematics, Nanjing
University, Nanjing 210093, People's Republic of China}
\email{zwsun@nju.edu.cn}

\keywords{Determinants, Legendre symbols, quadratic residues modulo primes.
\newline \indent 2020 {\it Mathematics Subject Classification}. Primary 11A15, 11C20; Secondary 15A15.
\newline \indent The second author is supported by the Natural Science Foundation of China (grant 12371004).}

\begin{abstract} Let $p$ be an odd prime and let $(\frac{\cdot}p)$ be the Legendre symbol.
In this paper, we study the determinant
$$\det\left[\left(\f{j^2-k^2}p\right)+\left(\f{jk}p\right)w\right]_{\delta\ls j,k\ls (p-1)/2}$$
with $\delta\in\{0,1\}$. For example, we prove that
the determinant does not depend on $w$ if $p\equiv3\pmod4$ and $\delta=0$.
\end{abstract}
\maketitle

\section{Introduction}
\setcounter{lemma}{0}
\setcounter{theorem}{0}
\setcounter{corollary}{0}
\setcounter{remark}{0}
\setcounter{equation}{0}

As usual, for a square matrix $A=[a_{jk}]_{1\ls j,k\ls n}$ over a commutative ring with identity,  we use $|A|$ or $|a_{jk}|_{1\ls j,k\ls n}$
to denote its determinant.

Let $p$ be an odd prime, and let $(\frac{.}{p})$ be the Legendre symbol.
Motivated by the works of R. Chapman's work \cite{C1}
and M. Vsemirnov \cite{V12,V13}, for any integer $d$ Z.-W. Sun \cite{S19} introduced the determinants
$$S(d,p)=\bg|\l(\f{j^2+dk^2}p\r)\bg|_{1\ls j,k\ls (p-1)/2}$$
and $$T(d,p)=\bg|\l(\f{j^2+dk^2}p\r)\bg|_{0\ls j,k\ls (p-1)/2},$$
and showed that
\begin{equation}\label{ST} S(d,p)=\begin{cases}\f2{p-1}T(d,p)&\t{if}\ (\f dp)=1,\\0&\t{if}\ (\f dp)=-1,
\end{cases}\ \t{and}\
\l(\f{T(d,p)}p\r)=\begin{cases}(\f{2}p)&\t{if}\ (\f dp)=1,
\\1&\t{if}\ (\f dp)=-1.\end{cases}
\end{equation}
In 2018 Sun \cite{S18} conjectured that $-S(1,p)$ is an integer square provided $p\eq3\pmod4$,
this was finally confirmed by D. Krachun based on an initial idea of M. Alekseyev via quadratic Gauss sums.

Recently, Sun \cite{S24} studied determinants whose entries are linear combinations of Legendre symbol
and posed many conjectures in this direction. For example, Sun \cite[Remark 1.1]{S24}
conjectured that
\begin{equation}\l|x+\l(\f{j^2+k^2}p\r)+\l(\f{j^2-k^2}p\r)\r|_{1\ls j,k\ls (p-1)/2}
=\l(\f{p-1}2x-1\r)p^{(p-3)/4}
\end{equation}
for any prime $p\eq3\pmod4$, and this was later confirmed by J. Li and H.-L. Wu \cite{LW}
via Jacobi sums over finite fields.

Motivated by the above work, we establish the following new result.

\begin{theorem} \label{conj} Let $p$ be an odd prime. For $\da\in\{0,1\}$, define
$$D_p^{(\da)}(w)=\l|\l(\f{j^2-k^2}p\r)+\l(\f {jk}p\r)w\r|_{\da\ls j,k\ls(p-1)/2}-\l|\l(\f{j^2-k^2}p\r)\r|_{\da\ls j,k\ls(p-1)/2}.$$

{\rm (i)} When $p\eq1\pmod4$,
there is a positive odd integer $c_p$  such that 
\begin{equation}\label{Dp0}
D_p^{(0)}(w)=\f{p-1}2D_p^{(1)}(w)=-w\l(\f{p-1}2c_p\r)^2.
\end{equation}

{\rm (ii)} Assume that $p\eq3\pmod4$. Then we have
\begin{equation} D_p^{(0)}(w)=0.
\end{equation}
 Also,
there is a positive odd integer $c_p$ such that
\begin{equation}\label{Dp1}
D_p^{(1)}(w)=w\(c_p\sum_{k=1}^{(p-1)/2}\(\f{k}p \)\)^2.
\end{equation}
\end{theorem}
\begin{remark}  For any prime $p>3$ with $p\eq3\pmod4$, it is known (cf. \cite[Chap. 5, Sec. 4]{BS} or \cite[p.\,238]{IR}) that
$$\sum_{k=1}^{(p-1)/2}\l(\f kp\r)=\l(2-\l(\f 2p\r)\r)h(-p),$$
where $h(-p)$ denotes the class number of the imaginary quadratic field $\Q(\sqrt{-p})$.
\end{remark}

 Our numerical computation indicates that
$$c_5=1,\,c_{13}=3,\,c_{17}=21,\,c_{29}=83,\,c_{37}=9095,\,c_{41}=98835,\,c_{53}=4689023$$
and
$$c_3=1,\,c_{7}=1,\,c_{11}=1,\,c_{19}=17,\,c_{23}=1,\,c_{31}=33,\,c_{43}=67119,\,c_{47}=1870591.$$
Based on this, we formulate the following conjecture.

\begin{conjecture} For any odd prime $p$, the Jacobi symbol $(\f p{c_p})$ is equal to $1$.
\end{conjecture}

We are going to prove parts (i) and (ii) of Theorem \ref{conj} in Sections 2 and 3 respectively.

Throughout this paper, for a matrix $A$ we use $A^T$ to denote the transpose of $A$.
For convenience, when a prime $p>3$ is fixed, we use the notations
\begin{equation}A_{\da}=\l[\l(\f{j^2-k^2}p\r)\r]_{\da\ls j,k\ls(p-1)/2}\ \ (\da=0,1,2)
\end{equation}
and
\begin{equation}\mathbf{u}_\da=\l[\(\f {\da}p\),\(\f {\da+1}p\),\cdots,\(\f {(p-1)/2}p\)\r]^T\ \ (\da=0,1).
\end{equation}

\section{Proof of the First Part of Theorem \ref{conj}}
\setcounter{lemma}{0}
\setcounter{theorem}{0}
\setcounter{corollary}{0}
\setcounter{remark}{0}
\setcounter{equation}{0}

A well known theorem of Jacobsthal (cf. \cite[Theorem 6.2.9]{BEW}) states that for any prime $p\eq1\pmod4$ we have $p=J(s)^2+J(t)^2$ for any $s,t\in\Z$ with $(\f{st}p)=-1$, where
$$J(k)=\sum_{x=1}^{(p-1)/2}\l(\f {x(x^2+k)}p\r)=\f12\sum_{x=0}^{p-1}\l(\f{x(x^2+k)}p\r)\ \ \t{for}\  k\in\Z.$$

\begin{lemma} \label{Lem-W}
Let $p$ be a prime with $p\eq1\pmod4$. Then $-S(-1,p)/J(-1)$ is an odd square.
\end{lemma}
\begin{proof} Write $p=a^2+b^2$ with $a,b\in\Z$ and $a\eq1\pmod4$.
Then $J(-1)=-(-1)^{(p-1)/4}a$ by \cite[Theorem 6.2.9]{BEW}. Also,
$(-1)^{(p-1)/4}S(-1,p)/a$ is an integer square by \cite[Theorem 3]{Wu}.
Thus $-S(-1,p)/J(-1)$ is an integer square.

Now it suffices to to show that $S(-1,p)\eq1\pmod2$. Clearly,
$$S(-1,p)=\sum_{\sigma\in S_{(p-1)/2}}\sign(\sigma)\prod_{j=1}^{(p-1)/2}\l(\f{j^2-\sigma(j)^2}p\r)$$
has the same parity with the derangement number
$$D_{(p-1)/2}=\l|\l\{\sigma\in S_{(p-1)/2}:\ \sigma(j)\not=j\ \t{for all}\ j=1,\ldots,\f{p-1}2\r\}\r|.$$
It is well known (cf. \cite[p.\,67]{St}) that for any positive integer $n$ we have
\begin{equation}\label{Dn} D_n=\sum_{k=0}^n(-1)^k\f{n!}{k!}\eq(-1)^n+(-1)^{n-1}n\eq n+1\pmod2.
\end{equation}
So
$$S(-1,p)\eq D_{(p-1)/2}\eq1\pmod 2.$$
This concludes our proof.
\end{proof}

The following useful lemma can be found in \cite{Vr}.

\begin{lemma}[The Matrix-Determinant Lemma]\label{le6}
Let $A$ be an $n\times n$ matrix over a field $F$, and let $\mathbf{u},\mathbf{v}$ be two $n$-dimensional column vectors with entries in $F$. Then
\begin{equation*}
\l| A+\u\v^T\r|=\l| A\r|+ \v^T \adj(A)\u,
\end{equation*}
where $\adj(A)$ is the adjugate matrix of $A$.	
\end{lemma}

{\noindent{\tt Proof of Theorem \ref{conj}(i)}. Suppose that $p\eq1\pmod4$. Let $\da\in\{0,1\}$.
Note that $A_{\da}$ is invertible since
\begin{equation}\label{A01}|A_0|=\f{p-1}2|A_1|\not\eq0\pmod p
\end{equation}
by \eqref{ST}.
In view of Lemma \ref{le6}, we have
\begin{equation}\label{Dp}D_p^{(\da)}(w)=\l|A_\da+w\u_\da\u_\da^T \r|-\l|A_\da \r|=w\u_\da^T\adj(A_\da)\u_\da=w\u_\da^T|A_{\da}|A_{\da}^{-1}\u_\da.
\end{equation}

Let $j\in\{1,\ldots,(p-1)/2\}$. For $k=1,\ldots,(p-1)/2$ let $r_j(k)$ be the unique integer $r\in\{1,\ldots,(p-1)/2\}$ with $jk$ congruent to $r$ or $-r$ modulo $p$.
Then
\begin{align*}
&\ \sum_{k=\da}^{(p-1)/2}\(\f {j^2-k^2}p\)\(\f {k}p\)
\\=&\ \sum_{k=1}^{(p-1)/2}\(\f {j^2-r_j(k)^2}p\)\(\f {r_j(k)}p\)
=\ \sum_{k=1}^{(p-1)/2}\(\f {j^2(1-k^2)}p\)\(\f {jk}p\)\\
=&\ \(\f {j}p\)\sum_{k=1}^{(p-1)/2}\(\f {k(k^2-1)}p\)=J(-1)\(\f {j}p\).
\end{align*}
Note also that
$$\sum_{k=0}^{(p-1)/2}\l(\f{0^2-k^2}p\r)\l(\f kp\r)=\f12\sum_{k=1}^{p-1}\l(\f kp\r)=0=J(-1)\l(\f 0p\r).$$
Therefore
$$A_\da\mathbf{u}_\da=J(-1)\mathbf{u}_\da$$
and hence
\begin{equation} \label{A-1}A_{\da}^{-1}\mathbf{u}_\da=J(-1)^{-1}\mathbf{u}_\da.
\end{equation}

Combining \eqref{Dp} and \eqref{A-1}, we obtain
$$D_p^{(\da)}(w)=w\u_\da^T|A_{\da}|J(-1)^{-1}\mathbf{u}_\da=\f{|A_{\da}|}{J(-1)}\times\f{p-1}2w.$$
Thus, with the aid of \eqref{A01}, we see that
$$D_p^{(0)}(w)=\f{p-1}2D_p^{(1)}(w)=-w\l(\f{p-1}2 c_p\r)^2,$$
where $c_p=\sqrt{-|A_1|/J(-1)}$ is a positive odd integer by Lemma \ref{Lem-W}.
This concludes the proof.
\qed

\section{Proof of the Second Part of Theorem \ref{conj}}
\setcounter{lemma}{0}
\setcounter{theorem}{0}
\setcounter{corollary}{0}
\setcounter{remark}{0}
\setcounter{equation}{0}

Let $p$ be an odd prime. By \cite[Theroem 2.1.2]{BEW}, for any $b,c\in\Z$ we have
\begin{equation}\label{le3.1}
\sum_{x=0}^{p-1}\(\f{x^2+bx+c}p \) =
\begin{cases}
-1 & \t{\rm if } p\nmid b^2-4c, \\
p-1 & \t{\rm if } p\mid b^2-4c.
\end{cases}
\end{equation}

\begin{lemma}\label{le3.2}
Let $p$ be a prime with $p\eq3\pmod4$.

{\rm (i)} For each $j\in\{1,\ldots,(p-1)/2\}$, we have
$$\sum_{k=1}^{(p-1)/2}\l(\f{j^2-k^2}p \r) =0.$$

{\rm (ii)} If $p>3$, then
$$\l|\l(\f{j^2-k^2}p\r)\r|_{2\ls j,k\ls(p-1)/2}\eq 1 \pmod2.$$
\end{lemma}
\begin{proof} (i) Let $j\in\{1,\ldots,(p-1)/2\}$. In view of \eqref{le3.1}, we have
$$\sum_{k=1}^{\frac{p-1}{2}}\(\f{j^2-k^2}p \) =\frac{1}{2}\sum_{k=1}^{p-1}\(\f{j^2-k^2}p \)=-\frac{1}{2}\(\sum_{k=0}^{p-1}\(\f{k^2-j^2}p \)+1\)=0.$$

(ii) Suppose $p>3$. Observe that
\begin{align*}\l|\l(\f{j^2-k^2}p\r)\r|_{2\ls j,k\ls(p-1)/2}
&=\sum_{\sigma\in S_{(p-3)/2}}\sign(\sigma)\prod_{j=1}^{(p-3)/2}\l(\f{(j+1)^2-(\sigma(j)+1)^2}p\r)
\\&\eq\l|\l\{\sigma\in S_{(p-3)/2}:\ \sigma(j)\not=j\ \t{for all}\ j=1,\ldots,\f{p-3}2\r\}\r|
\\&=D_{(p-3)/2}\eq \f{p-3}2+1\eq1\pmod 2
\end{align*}
with the aid of \eqref{Dn}. This ends the proof. \end{proof}

An $n\times n$ matrix $A$ over a commutative ring with identity is said to be skew-symmetric if $A^T=-A$.
The following classical result can be found in \cite[Prop.\,2.2]{K04}.

\begin{lemma}[Cayley's theorem] \label{le3.3}
If $A$ is a skew-symmetric matrix over $\Z$ and it is of even order, then $|A|$ is an integer square.
\end{lemma}

\noindent{\tt Proof of Theorem \ref{conj}(ii)}. Recall that
$$A_0=\l[\l(\f{j^2-k^2}p\r)\r]_{0\ls j,k\ls(p-1)/2}
\ \t{and}\ \adj(A_0)=[A_{kj}]_{0\ls j,k\ls(p-1)/2},$$
where $A_{kj}$ is the cofactor of the $(k,j)$-entry in the matrix $A_0$.
As $(\f{-1}p)=-1$ and $(p+1)/2\eq0\pmod2$, it is easy to verify that
$$A_0^T=-A_0\ \ \t{and}\ \ (\adj(A_0))^T=-\adj(A_0).$$
Since
$$\u_0^T\adj(A_0)\u_0=(\u_0^T\adj(A_0)\u_0)^{T}=\u_0^T(\adj(A_0))^T\u_0=-\u_0^T\adj(A_0)\u_0,$$
using Lemma \ref{le6} we deduce that
$$D_p^{(0)}(w)=w\u_0^T\adj(A_0)\u_0=0.$$

It remains to handle the case $\da=1$. When $p=3$, \eqref{Dp1} holds trivially with $c_3=1$.

Now we assume  $p>3$. By Lemma \ref{le3.2}, $|A_1|=0$ but $|A_2|\neq0.$ So
$$\rank(A_1)=\frac{p-3}{2}$$ and
$$\rank(\ker(A_1))=\rank(\adj(A_1))=1,$$
where $\ker(A_1)$ is the set
$$\{[x_1,x_2,\ldots,x_{(p-1)/2}]^T:\ x_i\in\R\ \t{and}\ A_1[x_1,x_2,\ldots,x_{(p-1)/2}]^T=[0,0,\ldots,0]^T\}.$$
As $$A_1[1,1,\cdots,1]^T=[0,0,\cdots,0]^T$$
by Lemma \ref{le3.2}(i), we have
$$\ker(A_1)=\{\lambda[1,1,\cdots,1]^T :\ \lambda\in\R\}.$$
Note that $A_1^T=-A_1$ and $(\adj(A_1))^T=\adj(A_1)$. Also,
$$A_1\adj(A_1)=|A_1|I_{(p-1)/2}=\mathbf{0},$$
where $I_n$ to denote the identity matrix of order $n$.
Thus, for any $j,k\in\{1,\ldots,(p-1)/2\}$ we have
$$|A_2|=\adj(A_1)_{1,1}=\adj(A_1)_{j,1}=\adj(A_1)_{1,j}=\adj(A_1)_{k,j}=\adj(A_1)_{j,k},$$
 where $\adj(A_1)_{j,k}$ denotes the $(j,k)$-entry of the matrix $\adj(A_1)$. Consequently
\begin{equation}\label{3.1}
\adj(A_1)=|A_2|\mathbf{1}
\end{equation}
where $\mathbf{1}$ denotes the matrix of order $(p-1)/2$ whose entries are all $1$.

Combining (\ref{3.1}) with Lemma \ref{le6} and Lemma \ref{le3.3} , we obtain
\begin{align*}\label{3.2}
D_p^{(1)}(w)&=w\u_1^T\adj(A_1)\u_1
=w\u_1^T|A_2|\mathbf{1}\u_1\\
&=w|A_2|\(\sum_{j=1}^{(p-1)/2}\l(\f{j}p \r)\)^2
=w\(c_p\sum_{j=1}^{(p-1)/2}\l(\f{j}p \r)\)^2.
\end{align*}
where $c_p=\sqrt{|A_2|}$ is a nonnegative integer by Lemma \ref{le3.3}.
As $|A_2|$ is odd, so is $c_p$.

In view of the above, we have completed the proof of Theorem \ref{conj}(ii).
\qed

\medskip

\end{document}